\documentclass[a4paper,11pt]{amsart}

\usepackage{amstext,amsfonts,amsthm,graphicx,amssymb,amscd,epsfig}
\usepackage{amsmath}
\usepackage[ansinew]{inputenc}    

\usepackage{mathrsfs}

\theoremstyle{definition}
\newtheorem{definition}{Definition}[section]
\newtheorem{ex}[definition]{Example}
\newtheorem{rem}[definition]{Remark}

\theoremstyle{plain}
\newtheorem{prop}[definition]{Proposition}
\newtheorem{lem}[definition]{Lemma}
\newtheorem{coro}[definition]{Corollary}
\newtheorem{teo}[definition]{Theorem}

\newfont{\bbb}{msbm10 scaled\magstephalf}     
\def\C{\mathbb C}
\def\K{\mathbb K}
\def\R{\mathbb R}

\def\N{\mbox{\bbb N}}

\def\R{\mbox{\bbb R}}

\def\O{\mathcal O}
\def\A{\mathscr A}

\def\Aecod{\operatorname{\mbox{$\A_e$-cod}}}

\def\Rcod{\operatorname{\mbox{$\mathscr R$-cod}}}

\def\Kcod{\operatorname{\mbox{$\mathscr K$-cod}}}

\usepackage[usenames]{color}

\newcommand{\tae}{T\A_e}

\newcommand{\nae}{N\A_e}

\newcommand{\tR}{T\mathscr{R}}

\newcommand{\tK}{T\mathscr{K}}
\newcommand{\fpdat}[3]{\left. \frac{\partial #1}{\partial #2} \right|_{#3}}
\newcommand{\dpar}[2]{\frac{\partial #1}{\partial #2}}
\def\m{\mathfrak{m}}

\title{Simplicity of Augmentations of Codimension 1 germs and by Morse functions}

\author{I. Breva Ribes, R. Oset Sinha}

\date{}

\address{Departament de Matem\`atiques,
Universitat de Val\`encia, Campus de Burjassot, 46100 Burjassot,
Spain}

\email{raul.oset@uv.es}

\email{igbreri@alumni.uv.es}

\thanks{Work of R. Oset Sinha partially supported by Grant PGC2018-094889-B-100 funded by MCIN/AEI/ 10.13039/501100011033 and by ``ERDF A way of making Europe"}


\begin{document}
\begin{abstract}
We study the simplicity of map-germs obtained by the operation of augmentation and describe how to obtain their versal unfoldings.
When the augmentation comes from an $\A_e$-codimension 1 germ or the augmenting function is a Morse function, we give a complete characterisation for simplicity.
These characterisations yield all the simple augmentations in all explicitly obtained  classifications of $\A$-simple monogerms except for one ($F_4$ in Mond's list from $\C^2$ to $\C^3$). 
Moreover, using our results we produce a list of corank 1 simple augmentations from $\C^4$ to $\C^4$.
\end{abstract}

\maketitle

\section{Introduction}

Since the very beginning the study of singularities, both of maps and spaces, has been linked to the concept of simplicity. 
Heuristically speaking, an object is said to be simple for a certain equivalence relation if any deformation of it lies in a finite number of equivalence classes. 
When studying singularities of maps, simplicity seems to be a reasonable criteria for when to stop classifying singularities. However, the lack of simplicity and the modality of a map-germ can be interesting by themselves. 
In \cite{arnoldnf} Arnol'd established the celebrated classification $A_k$ ($k\geq 1$), $D_k$ ($k\geq 4$), $E_6$, $E_7$ and $E_8$ of simple function germs (under $\mathscr R$-equivalence, i.e. changes of variables in the source) and classified the unimodal and bimodal germs. 
The modality is the dimension of the moduli space obtained by the quotient of the bifurcation diagram under the equivalence relation. 
In fact, Gabri\`{e}lov obtained in \cite{gabrielov} a relation between the lack of simplicity and the topology of the deformations and of the bifurcation diagrams by proving that the modality of a function germ is equal to the dimension of the $\mu$-constant stratum, i.e. the stratum in the bifurcation diagram for which the Milnor number is equal to the Milnor number of the function corresponding to the origin in the parameter space of the versal unfolding.

In general, when working with map-germs and $\mathscr A$-equivalence (changes of variables in source and target), it is not easy at all to determine when a certain map-germ is simple or not and to compute its modality. See section \ref{notation} for a precise definition of $\mathscr R$ and $\mathscr A$-simplicity. These difficulties can be seen in the several classifications of simple monogerms that have been carried out for different dimensions of the source and target $(n,p)$. For example, Rieger classified simple germs in $(2,2)$ (\cite{rieger2en2}), Mond in $(2,3)$ (\cite{mond}), Marar and Tari in $(3,3)$ (\cite{marartari}) and Houston and Kirk in $(3,4)$ (\cite{houstonkirk}).
It is impossible not to notice that most of the germs in these lists are augmentations of germs in lower dimensions.

\begin{definition}
For $\K=\R$ or $\C$, let $f\colon(\K^n,0)\to(\K^p,0)$ be a smooth map-germ that admits a 1-parameter stable unfolding, i.e., a 1-parameter family of map-germs $F(x,\lambda) = (f_\lambda(x),\lambda)$ with $f_0 = f$ which is stable as a map-germ.

Let $g\colon(\K^d,0)\to(\K,0)$ be a smooth function.
The augmentation of $f$ by $F$ via $g$ is defined as the $\A$-equivalence class of the germ given by:
$$A_{F,g}(f)(x,z) = (f_{g(z)}(x),z)$$
We say that $f$ is the augmented map and $g$ is the augmenting function.
\end{definition}

For example, consider the germ $(y^2,y^3)$ and its 1-parameter stable unfolding $(y^2,y^3+\lambda y,\lambda)$. The augmentations by $g(z)=z^{k+1}$ yield the family $S_k=(y^2,y^3+z^{k+1}y,z)$ in Mond's list.

A further analysis shows that, in fact, all of these simple augmentations (except for $F_4$ in Mond's list) come from cases where the augmented map has $\mathscr A_e$-codimension 1 or the augmenting function is a Morse function. It is therefore natural to study when an augmentation yields a simple map-germ. It is clearly not enough for the augmented map and the augmenting function to be simple: no augmentation of $(y^2,y^5)$ via the augmenting function $g(z)=z^4$ appears in Mond's list, and therefore none of them are simple.

In this paper we give the first steps in order to understand the simplicity of augmentations. 
In Section \ref{clases} we prove that in the case $g$ is $\mathscr R$-equivalent to a quasihomogeneous function, then the $\mathscr A$-equivalence class of the augmentation is uniquely determined by the $\mathscr R$-equivalence class of $g$. 
We also prove that if the augmented map has codimension 1, then the augmentation is independent of the choice of 1-parameter stable unfolding. In Section \ref{versal} we obtain a versal unfolding for any augmentation where the augmenting function is equivalent to a quasihomogeneous function. This is interesting in itself. 
For instance, in Example \ref{ex_versals} we give the versal unfolding for any map-germ from $\C^2$ to $\C^3$ (simple or not) of type $(x,y^2,y^{2k+1} + x^ly)$, $k,l\geq 1$. 
In Section \ref{simplicity} we give our main results, we prove that an augmentation where the augmented map has $\mathscr A_e$-codimension 1 is simple if and only if the augmenting function is simple, and that an augmentation where the augmenting function is a Morse function is simple if and only if the augmented map is simple. 
As mentioned above, this justifies the simplicity of the augmentations of all explicitly obtained classifications of simple monogerms except for $F_4$ in Mond's list. 
Notice that besides the classifications mentioned above, there are also known classifications for curves (i.e $n=1$) or for $(n,2n)$ (\cite{riegerpop}). 
However, the former cannot be augmentations due to the dimension of the source space and the latter are not augmentations either because the only stable germs are the immersions.
Goryunov also classified simple map germs when $n\geq p$, but no explicit lists of normal forms are provided, and so further analysis would be needed to find augmentations in these dimensions. 
The characterisations given in this paper also give a method to produce examples of finitely determined non-simple germs with any number of modal parameters.
 
Finally, in Section \ref{4en4} we produce a list of corank 1 simple augmentations from $\C^4$ to $\C^4$ using our results and, based on the classifications of \cite{rieger2en2} and \cite{marartari}, we conjecture that these are all possible simple augmentations of corank 1.

\emph{Acknowledgements:} The authors thank the Singularity Group in the Universitat de Val\`{e}ncia and M.A.S. Ruas for helpful discussions.

\section{Notation}\label{notation}
Let $\K$ be either $\R$ or $\C$.
We will denote by $\O_n$ the ring of germs of functions over $\K^n$, with maximal ideal $\m_n$, and by $\theta_n$ the space of germs of vector fields in $n$ components and $n$ variables.
If $f\colon(\K^n,0)\to(\K^p,0)$ is a smooth map-germ, we will write $\theta(f)$ to denote the space of germs of vector fields along $f$.
We will assume $(n,p)$ belongs to the nice dimensions: those where stable maps are dense.

Denote by $\operatorname{Diff}(\K^m,0)$ the group of germs of diffeomorphisms on $\K^m$.
Then $\A = \operatorname{Diff}(\K^n,0)\times \operatorname{Diff}(\K^p,0)$ and $\mathscr{R} = \operatorname{Diff}(\K^n,0)$ are groups that act over the set of map-germs $f\colon(\K^n,0)\to(\K^p,0)$ with the natural compositions.
The equivalence relations defined by the orbits of this actions are called $\A$-equivalence and $\mathscr{R}$-equivalence respectively.

Following the notation in the recent book \cite{nunomond}, and to be completely precise, we can define similar equivalence relations with the set of diffeomorphisms that do not necessarily preserve the base points in the source and target of the map-germs: these are called left-right and right equivalences, respectively.

Let $tf\colon\theta_n\to\theta(f)$ and $\omega f \colon\theta_p\to\theta(f)$ be the mappings defined by $tf(\xi) = df\circ\xi, \omega f(\eta) = \eta\circ f$.
Define the $\A_e$-tangent space of $f$ and the $\A_e$-normal space of $f$, respectively, as:
\begin{equation*}
\begin{aligned}
\tae f  & = tf(\theta_n) + \omega f (\theta_p), \\
\nae f  & = \frac{\theta(f)}{tf(\theta_n) + \omega f (\theta_p)}
\end{aligned}
\end{equation*}
The $\A_e$-codimension of $f$ is defined as $\Aecod(f) = \dim_\K(\nae f)$.
We will say that $f$ is stable if $\Aecod(f) = 0$.
A precise definition of stability using unfoldings can be found in \cite{nunomond}.
We define, for functions $g\in \O_d$, the following $\mathscr R$-tangent spaces:
\begin{equation*}
\begin{aligned}
T\mathscr{R}g  & = \m_d\cdot Jg, \\
T\mathscr{R}_e g & = Jg
\end{aligned}
\end{equation*}
with $Jg$ the $\O_d$-ideal generated by the partial derivatives of $g$.
Codimensions are defined in the same way as for $\A$.

An $m$-parameter unfolding of $f$ is a map-germ $F\colon(\K^n\times\K^m,0)\to(\K^p\times\K^m,0)$ with a representative of the form $F(x,\lambda) = (f_\lambda(x),\lambda)$ and $f_0 \equiv f$.
We say that $F$ is a versal unfolding if for any other $a$-parameter unfolding of $f$ there exist a differentiable map-germ $\alpha\colon(\K^a,0)\to(\K^m,0)$ and two unfoldings of the identity map, $\Psi,\Phi$ such that:
$$\alpha^*F = \Psi\circ G\circ \Phi$$
with $\alpha^*F(x,\mu) = (f_{\alpha(\mu)},\mu)$.
This definition of versality is equivalent to checking that $F$ satisfies the following condition:
\begin{equation*}
\tae f + \text{Sp}_\K\left\lbrace\fpdat{f_\lambda}{\lambda_1}{\lambda=0},\ldots,\fpdat{f_\lambda}{\lambda_m}{\lambda=0}\right\rbrace = \theta(f)
\end{equation*}

\begin{definition}
A map-germ is said to be $\A$-simple if there are a finite number of left-right equivalence classes such that, given its versal unfolding $F\colon(\K^n\times\K^d,0)\to(\K^p\times\K^d,0)$, it admits a representative $F\colon U\to V$ with $U\subseteq \K^n\times\K^d$, $V\subseteq \K^p\times\K^d$ small enough neighbourhoods of the origin, such that $F(x,\lambda) =(f_\lambda(x),\lambda)$ and for every $\lbrace (x_1,\lambda),\ldots,(x_r,\lambda)\rbrace \subseteq U$ with $F(x_1,\lambda) = \ldots =  F(x_r,\lambda) = (y,\lambda)$, the equivalence class of $f_\lambda\colon(\K^n,\lbrace x_1,\ldots,x_r \rbrace)\to(\K^p,y)$ lies in one of those classes.
\end{definition}
Similar definitions of versality and simplicity can be defined for $\mathscr{R}$-equivalence.
For a detailed explanation of stability and versality we refer to Chapter 3 from \cite{nunomond}.

\begin{definition}
Let $f\colon(\K^n,0)\to(\K^p,0)$ be a smooth map-germ, an OPSU for $f$ is a 1-parameter stable unfolding $F\colon(\K^n\times\K,0)\to(\K^p\times\K,0)$ of the form $F(x,\lambda) = (f_\lambda(x),\lambda)$, with $f_0=f$.
\end{definition}

Throughout the paper, we will assume that the augmenting function $g$ has an isolated singularity.
The following result by Houston provides a simple way of estimating the $\A_e$-codimension of an augmentation:
\begin{teo}[Theorem 4.4 in \cite{houston}]\label{thm_houston_ineq}
Let $f\colon(\K^n,0)\to(\K^p,0)$ be a smooth map-germ that admits an OPSU $F$.
Let $g\colon(\K^d,0)\to(\K,0)$ be a smooth function.
$$\Aecod (A_{F,g}(f))\geq \Aecod (f) \tau(g)$$
with equality if $g$ is quasihomogeneous or $F$ is substantial.
\end{teo}
An OPSU $F(x,\lambda)=(f_\lambda(x),\lambda)$ is said to be substantial if it admits a liftable vector field $\eta=(\eta_1,\ldots,\eta_{p+1})\in\theta_{p+1}$ such that $\eta_{p+1}\circ F(x,\lambda) = \lambda$.
Not all OPSUs are substantial.
For instance, the $\A_e$-codimension 2 map-germ $(x,y^4+xy^2+x^2y)$  admits the following OPSU:
$$(x,y^4 + xy^2 + x^2y + \lambda y, \lambda)$$
Using SINGULAR, one can check that it is not substantial.

Recall that if $g\colon(\K^d,0)\to(\K,0)$ is a smooth function, $\tau(g) = \operatorname{dim}_\K\frac{\O_d}{ \langle g\rangle + Jg}$ and $\mu(g) = \operatorname{dim}_\K\frac{\O_d}{Jg}$.
We say that $g$ is quasihomogeneous if there exists some weights $w_1,\ldots,w_d\in \N$ and some degree $w\in\N$ such that for any $\lambda\in \K$:
$$g(\lambda^{w_1}z_1,\ldots,\lambda^{w_d}z_d) = \lambda^wg(z)$$

\section{Characterising Equivalence of Augmentations}\label{clases}

From here on $f\colon(\K^n,0)\to(\K^p,0)$ is a smooth map-germ that admits an OPSU $F(x,\lambda) =(f_\lambda(x),\lambda)$ and $g\colon(\K^d,0)\to(\K,0)$ is a smooth function.
First we prove that the process of augmenting via $\mathscr R$-equivalent functions returns $\A$-equivalent augmentations.

\begin{prop} \label{prop_requiv_augm_aequiv}
If $g_1 \sim_\mathscr{R} g_2$, then $A_{F,g_1}(f) \sim_\A A_{F,g_2}(f)$
\end{prop}
\begin{proof}
Let $\phi\colon(\K^d,0)\to(\K^d,0)$ be a germ of diffeomorphism such that:
$$g_2 = g_1 \circ \phi$$
Define the following germs of diffeomorphism:
\begin{align*}
\Phi & \colon(\K^n\times \K^d,0)\to(\K^n\times \K^d,0)\\
\Psi & \colon(\K^p\times \K^d,0)\to(\K^p\times \K^d,0)
\end{align*}
given by $\Phi(x,z) = (x,\phi(z))$, $\Psi(X,Z) = (X,\phi^{-1}(Z))$. Then:
\begin{equation*}
\begin{split}
\Psi\circ A_{F,g_1}(f)\circ\Phi(x,z) & = \Psi\circ A_{F,g_1}(f)(x,\phi(z)) \\
                                  & = \Psi(f_{g_1(\phi(z))}(x),\phi(z)) \\
                                  & = (f_{g_1(\phi(z))}(x),\phi^{-1}(\phi(z)))\\
                                  & = (f_{g_2(z)}(x),z) = A_{F,g_2}(f)(x,z)
\end{split}
\end{equation*}
so both augmentations are $\A$-equivalent via $\Phi, \Psi$.
\end{proof}

Notice that the same result holds using right and left-right equivalences.

This reduces considerably our task, since when studying the augmenting function it will be enough to check the augmentation via any representative in its $\mathscr{R}$-equivalence class.
Moreover, when a quasihomogeneous representative can be found we can determine exactly the $\A_e$-codimension of the augmentation:

\begin{coro}\label{coro_igualdad_houston}
If either $g$ is $\mathscr R$-equivalent to a quasihomogeneous function, or $F$ is substantial, then:
$$\Aecod\left(A_{F,g}(f)\right) = \Aecod (f) \tau(g)$$
\end{coro}
\begin{proof}
If $F$ is substantial, it is the case of Theorem \ref{thm_houston_ineq}.
Let $\tilde{g}$ be quasihomogeneous and $\mathscr R$-equivalent to $g$:
\begin{equation*}
\begin{aligned}
\Aecod(A_{F,g}(f)) & = \Aecod(A_{F,\tilde{g}}(f)) && \text{(by Proposition \ref{prop_requiv_augm_aequiv})} \\
 &  = \Aecod (f)\tau(\tilde{g}) && \text{(by Theorem \ref{thm_houston_ineq})} \\
 & = \Aecod (f)\tau(g) && \text{(since }g\sim_{\mathscr R} \tilde{g})
\end{aligned}
\end{equation*}
\end{proof}

In the complex case, a partial converse to Proposition \ref{prop_requiv_augm_aequiv} can be obtained in the case that both augmenting functions are $\mathscr R$-equivalent to a quasihomogeneous function.
To see this, first recall that two functions $g,g'\colon(\C^d,0)\to(\C,0)$ are $\mathscr K$-equivalent if there exists a germ of diffeomorphism over $\C^d\times \C$ of the form:
$$\theta(x,y) = (\phi(x),\psi(x,y))$$
such that $\phi$ is a diffeomorphism over $\C^d$, $\psi(x,0) \equiv 0$ and $(\phi(x),g(\phi(x))) = \theta(x,g'(x))$.
This is an equivalence relation, and we can define the following tangent spaces:
\begin{align*}
T \mathscr{K}_e g &= Jg + \langle g\rangle \\
T \mathscr{K} g &= \m_d\cdot Jg + \langle g\rangle
\end{align*}
Respective $\mathscr K$ and $\mathscr{K}_e$-codimensions can be defined analogously to those of $\A$ and $\mathscr R$.

\begin{prop}\label{prop_qh_kr}
Two complex functions, $\mathscr{R}$-equivalent to quasihomogeneous functions and with isolated singularity are $\mathscr K$-equivalent if and only if they are $\mathscr R$-equivalent.
\end{prop}
\begin{proof}
Since $\mathscr R$-equivalence entails $\mathscr K$-equivalence, one implication is trivial.

First we prove the other implication in the case that $g,g'\in \O_d$ are quasihomogeneous.
In this case, $g \in \m_d\cdot Jg$ and so $\tK g = \tR g$.
Similarly, $\tK g' = \tR g'$.
Now, since they have isolated singularity they are of finite $\mathscr R$-codimension, and therefore finitely $\mathscr R$-determined: we have $\Kcod(g) = \Rcod(g) < \infty$ and similarly $\Kcod(g') = \Rcod(g')<\infty$

Following the argument found in Theorem 7.1 from \cite{nunomond} or similarly in Theorem 3.2 from \cite{extranicedim}, both $\mathscr R$-orbits of $g$ and $g'$ are open on the same $\mathscr K$-orbit, all three of them being of the same finite codimension.
Therefore, since over $\C$ there can be only one open $\mathscr R$-orbit in the same $\mathscr K$-orbit, the $\mathscr R$-orbits of $g$ and $g'$ must coincide, and so they are $\mathscr R$-equivalent.

For the general case, assume $\bar{g},\bar{g}'\in\O_d$ are $\mathscr R$-equivalent to quasihomogeneous functions $g,g'\in\O_d$ respectively.
If $\bar{g}, \bar{g}'$ are $\mathscr{K}$-equivalent, $g,g'$ are also $\mathscr K$-equivalent.
Since they are quasihomogeneous, we now have that $g$ and $g'$ are $\mathscr R$-equivalent. Then:
$$\bar g \sim_\mathscr{R} g  \sim_\mathscr{R} g' \sim_\mathscr{R}\bar g'$$

\end{proof}

Therefore we can prove:

\begin{teo} \label{thm_r_iff_a}
Assume $g, g'\colon(\C^d,0)\to(\C,0)$ have an isolated singularity and both are $\mathscr{R}$-equivalent to quasihomogeneous functions.
Then:
$$ A_{F,g}(f) \sim_\A A_{F,g'}(f) \iff g\sim_{\mathscr R}g'$$
\end{teo}

\begin{proof}
One implication has been proven in Proposition \ref{prop_requiv_augm_aequiv}.
To check the other one, assume that $A_{F,g}(f) \sim_\A A_{F,g'}(f)$.
Proposition 3.4 from \cite{ORW2} ensures $g$ and $g'$ are $\mathscr K$-equivalent.
Therefore, by Proposition \ref{prop_qh_kr}, $g$ and $g'$ are $\mathscr R$-equivalent.
\end{proof}

\begin{lem}\label{lemma_leq_qh}
Let $g\colon(\C^d,0)\to(\C,0)$ be a smooth function with isolated singularity and $\mathscr R$-equivalent to a quasihomogeneous function, and let $\alpha\colon(\C,0)\to(\C,0)$ be a germ of diffeomorphism.
Then $\alpha\circ g$ is $\mathscr R$-equivalent to a quasihomogeneous function.
\end{lem}
\begin{proof}
It is enough to prove this result for the case that $g$ is quasihomogeneous.
Since $\alpha$ is a diffeomorphism, $J(\alpha\circ g) = Jg$, and so $\mu(g) = \mu(\alpha\circ g)$.
Now, since $\alpha\circ g$ can be expanded as a polynomial in $g$, we have that $\alpha\circ g \in \langle g\rangle \subseteq Jg$. This implies that $\tau(\alpha\circ g) = \tau(g) = \mu(g) = \mu(\alpha\circ g)$. Saito's Lemma (\cite{saito}) now ensures $\alpha\circ g$ is $\mathscr R$-equivalent to a quasihomogeneous function. 
\end{proof}

In \cite{robertamond}, the authors prove that when $\K = \C$ and $\Aecod(f) = 1$, augmentation via a quadratic function is independent of the chosen OPSU.
Their proof can be in fact extended to augmentation of an $\A_e$-codimension 1 germ via any function $\mathscr R$-equivalent to a quasihomogeneous function with the help of the above results:

\begin{prop}\label{thm_cod1_opsu}
If $\Aecod(f) = 1$, the augmentation of $f$ via a function $g$ that is $\mathscr R$-equivalent to a quasihomogeneous function does not depend on the chosen OPSU $F$: it depends on the $\A$-class of $f$ and the $\mathscr R$-class of $g$.
\end{prop}
\begin{proof}
By Proposition \ref{prop_requiv_augm_aequiv} it suffices to consider the case that $g$ is quasihomogeneous.
Let $F(x,t) = (f_t(x),t)$ and $\tilde{F}(x,t) = (\tilde{f}_t(x),t)$ be two different OPSUs of $f$.

Since $\Aecod(f) = 1$, they are also versal unfoldings: therefore, there must exist some germ of diffeomorphism $\alpha\colon(\C,0)\to(\C,0)$ and two unfoldings of the identity maps, $\Phi(x,t) = (\phi_t(x),t)$ and $\Psi(X,T) = (\psi_T(X),T)$ respectively such that:
\begin{equation}\label{eq_cod1_opsu_ind}
(f_{\alpha(t)}(x),t) = \alpha^* F(x,t) = \Psi\circ \tilde{F}\circ \Phi(x,t)= (\psi_t\circ \tilde{f}_t\circ\phi_t(x),t)
\end{equation}

By Lemma \ref{lemma_leq_qh}, $\alpha\circ g$ is $\mathscr R$-equivalent to a quasihomogeneous function. Since $\alpha\circ g$ is $\A$-equivalent to $g$ (and therefore $\mathscr K$-equivalent to $g$),  Propositon \ref{prop_qh_kr} now ensures there exists another diffeomorphism $\beta\colon(\C^d,0)\to(\C^d,0)$ such that $\alpha\circ g = g \circ \beta$.

By Proposition \ref{prop_requiv_augm_aequiv}, we have that:
$$A_{F,g}(f) \sim_{\A} A_{F,g\circ\beta}(f)$$
And by Equation (\ref{eq_cod1_opsu_ind}):
\begin{equation*}
\begin{split}
A_{F,g\circ\beta}(f)(x,z)
   & = (f_{g(\beta(z))},z)\\
   & = (f_{\alpha(g(z))},z)\\
   & = (\psi_{g(z)}\circ \tilde{f}_{g(z)}\circ \phi_{g(z)}(x),z) \\
   & = A_{\Psi,g}(\text{Id}_p)\circ A_{\tilde{F},g}(f) \circ A_{\Phi,g}(\text{Id}_n)
\end{split}
\end{equation*}
Which implies:
$$A_{F,g}(f) \sim_{\A} A_{\tilde{F},g}(f)$$

To check the second statement, consider $f$ and $h$ two $\A$-equivalent map-germs, and let $F(x,t) = (f_t(x),t)$ and $H(x,t) = (h_t(x),t)$ be two OPSUs of $f$ and $h$ respectively.
There must exist certain germs of diffeomorphism $\phi\colon(\C^n,0)\to(\C^n,0)$, $\psi\colon(\C^p,0)\to(\C^p,0)$ such that $f = \psi\circ h\circ \phi$.
Let $\Phi(x,t) = (\phi(x),t)$ and $\Psi(X,T) = (\psi(X),T)$, which are germs of diffeomorphism.
Note that $\Psi\circ H \circ \Phi (x,t) = (\psi\circ h_t \circ \phi(x),t)$ is an OPSU of $f$, and therefore:
$$A_{F,g}(f) \sim_\A A_{\Psi\circ H\circ \Phi, g}(f)$$
Now just check that:
$$A_{\Psi\circ H\circ \Phi, g}(f)(x,z) = (\psi\circ h_{g(z)}\circ \phi(x),z) = \Psi \circ A_{H,g}(h) \circ \Phi$$
We have $A_{F,g}(f) \sim_\A A_{H,g}(h)$.
\end{proof}

\section{The Versal Unfoldings}\label{versal}

Since our objective is checking when an augmentation is simple, the first step will be constructing an adequate versal unfolding.
We want to do this by finding a basis for the $\A_e$-normal space of the augmentation, and so the first results of this section are focused on finding some independent generators in the case that $\K=\C$.

Throughout the section we assume that $f\colon(\C^n,0)\to(\C^p,0)$ is a smooth map-germ with $\Aecod(f)=k$, that admits an OPSU $F(x,\lambda) = (f_\lambda(x),\lambda)$.
Let $g\colon(\C^d,0)\to(\C,0)$ be a smooth function with isolated singularity, and consider the augmentation $A_{F,g}(f)$.

For any $\gamma\in\theta(f)$, $(\gamma,0,\overset{d}\ldots,0)$ can be seen as a vector field along $A_{F,g}(f)$.
Moreover, we can multiply any of these vector fields by any function $\tau\in\O_d$ and they remain as vector fields along the augmentation.
We write the extended vector fields in this way:
\begin{align*}
\begin{bmatrix}
\tau(z)\gamma(x) \\
0
\end{bmatrix} & \in \theta(A_{F,g}(f))
\end{align*}
Notice that we can always take a monomial $\C$-basis of the quotient space $\frac{\O_d}{\langle g\rangle +Jg}$.

\begin{teo}\label{thm_basis_main}
Let $g\in\O_d$ and suppose $\tau_1,\ldots,\tau_r\in\O_d$ form a monomial $\C$-basis of the space $\frac{\O_d}{\langle g\rangle + Jg}$.
Assume that $\gamma\in\theta(f)$ is a vector field along $f$ such that, for some $s=1,\ldots,r$:
\begin{align*}
\begin{bmatrix}
\tau_s(z)\gamma(x) \\
0
\end{bmatrix}\in\tae (A_{F,g}(f))
\end{align*}
Then $\gamma\in\tae f$.
\end{teo}

\begin{proof}
Let $F$ be of the form $F(x,\lambda) = (\bar{f}(x,\lambda),\lambda)$ with $\bar f(x,0) = f(x)$.
Then, by Hadamard's Lemma (see Appendix C.1 in \cite{nunomond}), there must exist a smooth map-germ $\tilde f\colon(\K^n\times\K,0)\to(\K^p,0)$ such that:
\begin{equation}\label{eq_barf}
\bar f (x,\lambda) = f(x) + \lambda\tilde f(x,\lambda)
\end{equation}
Denote the augmentation of $f$ by $F$ via $g$ as:
$$A(x,z) = A_{F,g}(f)(x,z) = (\bar f(x,g(z)),z)$$
By hypothesis, there exist vector fields $\xi\in\theta_{n+d}$ and $\eta\in\theta_{p+d}$ such that:
\begin{align}\label{eq_in_taea}
\begin{bmatrix}
\tau_s\gamma \\
0
\end{bmatrix} = dA\circ\xi +\eta\circ A
\end{align}
We will prove our result by expanding each component from the right-hand side of this equation as polynomials in $z$ with coefficients in $\O_n$.
First, using the chain rule we check that:
\[
dA|_{(x,z)} = \left(\begin{array}{@{}cc@{}}
  d_x\bar{f}|_{(x,g(z))}
  & d_\lambda\bar f|_{(x,g(z))}\circ dg|_z \\
  0 &
  \operatorname{Id}_d
\end{array}\right)
\]
Here $d_x\bar f$ and $d_\lambda\bar f$ are the differential matrices of $\bar f$ with respect to $x$ and $\lambda$ respectively.
Set $\xi = (\xi_1,\ldots,\xi_{n+d})$ and $\eta=(\eta_1,\ldots,\eta_{p+d})$.
Denote $\xi_{1,\ldots,n}$ and $\xi_{n+1,\ldots,n+d}$ as the first $n$ and last $d$ components of $\xi$ respectively.
Define $\eta_{1,\ldots,p}$ and $\eta_{p+1,\ldots,p+d}$ analogously.
We can rewrite the first $p$ components in Equation (\ref{eq_in_taea}) as:
\begin{multline} \label{eq_expansion}
\tau_s(z)\gamma(x) = d_x\bar{f}|_{(x,g(z))}(\xi_{1,\ldots,n})(x,z) + d_\lambda\bar f|_{(x,g(z))}\circ dg|_z(\xi_{n+1,\ldots,n+d})(x,z) \\ + \eta_{1,\ldots,p}(\bar{f}(x,g(z)),z)
\end{multline}

Since $\tau_1,\ldots,\tau_r$ form a $\C$-basis of $\frac{\O_d}{\langle g\rangle + Jg}$, we can expand each component of $\xi$ as a polynomial over $z$ in the following way:
\begin{equation*}
\xi_i(x,z) = \sum_{j=1}^r\tau_j(z)\xi_i^{(j)}(x) + \tilde{\xi_i}(x,z)
\end{equation*}
with $\tilde{\xi_i}\in ( \langle g\rangle + Jg)\cdot\O_{n+d}$ and $\xi_i^{(j)}\in\O_n$, for each $i=1,\ldots,n+d$ and $j=1,\ldots,r$.
Now, using the expansion of $\bar f$ from Equation (\ref{eq_barf}), the first summand in Equation (\ref{eq_expansion}) can be rewritten as:
\begin{multline}\label{eq_xiexpansion}
d_x\bar{f}|_{(x,g(z))}(\xi_{1,\ldots,n}) = 
\sum_{j=1}^r\tau_j(z)\cdot d_xf|_x(\xi_{1,\ldots,n}^{(j)}(x)) + d_xf|_x(\tilde\xi_{1,\ldots,n})(x,z)\\
 + g(z)\cdot d_x\tilde f|_{(x,g(z))}(\xi_{1,\ldots,n})
\end{multline}
All the components in the last two summands of this equation are elements of $(\langle g\rangle + Jg)\cdot\O_{n+d}$.

We can expand $\eta$ in a similar way. 
Before composition with $A$, for each $i=1,\ldots,p+d$:
\begin{equation*}
\eta_i(X,Z) = \sum_{j=1}^r\tau_j(Z)\eta_i^{(j)}(X) + \tilde{\eta_i}(X,Z)
\end{equation*}
for some $\tilde{\eta}_i\in(\langle g\rangle + Jg)\cdot\O_{p+d}$ and some $\eta_i^{(j)}\in\O_p$ for $j=1,\ldots,r$.
After composing with $A$, this yields:
\begin{equation*}
\eta_i\circ A(x,z) = \sum_{j=1}^r\tau_j(z)\eta_i^{(j)}(f(x)+g(z)\tilde{f}(x,g(z))) + \tilde{\eta_i}\circ A(x,z)
\end{equation*}
Notice that $\tilde{\eta}_i\circ A\in( \langle g\rangle + Jg)\cdot\O_{n+d}$.
Considering $\eta_i^{(j)}(f(x)+\lambda\tilde{f}(x,\lambda))$ for each $j=1,\ldots,r$ as an element of $\O_{n+1}$, Hadamard's Lemma ensures there must exist a certain $\tilde{\eta}_i^{(j)}\in\O_{n+1}$  such that:
$$\eta_i^{(j)}(f(x)+\lambda\tilde{f}(x)) = \eta_i^{(j)}(f(x)) + \lambda\tilde{\eta}_i^{(j)}(x,\lambda)$$
And therefore:
\begin{equation}\label{eq_etaexpansion}
\eta_i\circ A(x,z) = \sum_{j=1}^r\tau_j(z)\eta_i^{(j)}(f(x)) + g(z)\sum_{j=1}^r\tilde{\eta}_i^{(j)}(x,g(z)) + \tilde{\eta_i}\circ A(x,z)
\end{equation}
Finally, we apply the expansions obtained in Equations (\ref{eq_xiexpansion}), (\ref{eq_etaexpansion}) to Equation (\ref{eq_expansion}) and look for the coefficient of $\tau_s$ in both sides.
Notice that anything multiplied by a function in the ideal $\langle g \rangle + Jg$ has coefficient $0$ for $\tau_s$, since it is a monomial with non-zero class over $\frac{\O_d}{\langle g\rangle + Jg}$.
In particular, all components in the second summand of (\ref{eq_expansion}) are multiplied by an element of $Jg$.
This step could also be applied to a non-monomial basis, but extra care should be taken when choosing representatives.
Gathering all the components, we get:
\begin{equation*}
\gamma(x) = \sum_{i=1}^n\dpar{f}{x_i}(x)\xi_i^{(s)}(x) + \eta_{1,\ldots,p}^{(s)}(f(x))\in\tae f
\end{equation*}
with $\eta_{1,\ldots,p}^{(s)} = (\eta_{1}^{(s)},\ldots,\eta_{p}^{(s)})$.

\end{proof}

As a corollary we obtain a quick method to construct an independent set of vector fields over $\nae \left(A_{F,g}(f)\right)$ in terms of those of $\nae f$.
When $g$ is $\mathscr R$-equivalent to a quasihomogeneous function, or $F$ is substantial, by Corollary \ref{coro_igualdad_houston} this set will in fact be a basis and so we will have a versal unfolding.
\begin{coro}\label{coro_augm_generadores}
Let $\gamma_1,\ldots,\gamma_k\in\theta(f)$ be a $\C$-basis for $\nae f$, and $\tau_1,\ldots,\tau_r$ a monomial $\C$-basis for $\frac{\O_d}{ \langle g\rangle +Jg}$.
The following vector fields are all non-zero and form an independent set over $\nae \left(A_{F,g}(f)\right)$:
$$\left\lbrace
\begin{bmatrix}
\tau_s(z)\gamma_m(x) \\
0
\end{bmatrix}
\right\rbrace_{\substack{m=1,\ldots,k \\ s = 1,\ldots,r}}
$$
\end{coro}
\begin{proof}
A vector field is equal to zero over $\nae \left(A_{F,g}(f)\right)$ if and only if it belongs to $\tae \left(A_{F,g}(f)\right)$.
By contradiction, assume for some $m\in\left\lbrace 1,\ldots,k\right\rbrace$ and $s\in\left\lbrace 1,\ldots,r\right\rbrace$:
\begin{align*}
\begin{bmatrix}
\tau_s(z)\gamma_m(x) \\
0
\end{bmatrix}\in\tae (A_{F,g}(f))
\end{align*}
Then, by Theorem \ref{thm_basis_main}, $\gamma_m \in \tae f$, which is absurd since it is a generator of a basis for $\nae f$.

To check that their classes are linearly independent, assume that there exist some $a_{m,s}\in\C$ not all zero, for $m=1,\ldots,k$ and $s=1,\ldots,r$, and some vector fields $\xi\in\theta_{n+d},\eta\in\theta_{p+d}$ such that:
\begin{align*}
\begin{bmatrix}
\sum_{m,s=1}^{k,r}a_{m,s}\tau_s\gamma_m \\
0
\end{bmatrix}\ = dA(\xi) + \eta\circ A
\end{align*}
with $A(x,z) = A_{F,g}(f)(x,z) = (\bar f(x,g(z)),g(z))$.
Take any $s'$ such that $a_{m,s'}$ is non-zero for some $m$.
Then, following the same notation as in the proof of Theorem \ref{thm_basis_main}, we can write:
\begin{multline*}
\sum_{m=1}^{k}a_{m,s'}\tau_{s'}(z)\gamma_m(x) = \sum_{m=1}^k\sum_{s\neq s'}^r a_{m,s}\tau_s(z)\gamma_m(x) + d_x\bar f|_{(x,g(z))}(\xi_{1,\ldots,n})\\  + d_\lambda\bar f|_{(x,g(z))}\circ dg|_z(\xi_{n+1,\ldots,n+d}) + \eta_{1,\ldots,p}\circ A
\end{multline*}
We can expand the components of $\xi$ and $\eta$ as in the proof of Theorem \ref{thm_basis_main} to look for the coefficients of $\tau_{s'}$ in the right-hand side of the equation. 
Since this coefficient is zero for the first summand, we can proceed as in the the last theorem to find some $\tilde{\xi}\in\theta_n$ and $\tilde\eta\in\theta_p$ such that:
$$\sum_{m=1}^k a_{m,s'}\gamma_m(x) = df(\tilde\xi)(x) + \tilde\eta\circ f(x)$$
which is absurd since $\gamma_m$ are taken to form a linearly independent  set as classes over $\tae f$.

\end{proof}

Finally, we obtain:

\begin{coro}\label{coro_nice_versal}
Let $f\colon(\C^n,0)\to(\C^p,0)$ be a smooth map-germ with $\Aecod(f) = k$ that admits an OPSU $F(x,t) = (f_t(x),t)$, and suppose $\gamma_1,\ldots,\gamma_k\in\theta(f)$ forms a $\C$-basis of $\nae f$ such that $\gamma_1 = \frac{\partial f_t}{\partial t}$ evaluated in $t=0$. 
Let $g\in\O_d$ and consider $\tau_1,\ldots,\tau_r$ a monomial $\C$-basis for $\frac{\O_d}{ \langle g\rangle + Jg}$.

If $g$ is $\mathscr R$-equivalent to a quasihomogeneous function, or if $F$ is substantial, $A_{F,g}(f)$ admits the following versal unfolding:
\begin{equation}\label{eq_nice_versal}
VA_{F,g}(f)(x,z,\lambda) = \left( f_{g(z) +\sum_{s=1}^r\lambda_{s,1}\tau_s(z)}(x)  + \sum_{s=1}^r\sum_{m=2}^k \lambda_{s,m}\tau_s(z)\gamma_m(x) ,z,\lambda \right)
\end{equation}
\end{coro}
\begin{proof}
Denote by $A^\lambda_{F,g}(f)$ the first $p+d$ components of $VA_{F,g}(f)$ for each $\lambda\in\C^{kr}$.
To check that $VA_{F,g}(f)$ is in fact a versal unfolding of $A_{F,g}(f)$ we only need to see that:
$$\nae\left(A_{F,g}(f)\right) = \operatorname{Sp}_\C\left\lbrace \fpdat{A^\lambda_{F,g}(f)}{\lambda_{s,m}}{\lambda = 0}\right\rbrace_{\substack{m=1,\ldots,k \\ s = 1,\ldots,r}}$$
Corollary \ref{coro_augm_generadores} now ensures that this set is in fact an independent set of non-zero elements of $\nae\left(A_{F,g}(f)\right)$.
And by Corollary \ref{coro_igualdad_houston}, $\Aecod \left(A_{F,g}(f)\right) = \Aecod(f)\tau(g) = kr$, which is exactly the number of vector fields in the set.
Therefore, it forms a basis, and so $VA_{F,g}(f)$ is a versal unfolding.
\end{proof}

We finish this section by giving some examples on how to calculate a versal unfolding for an augmentation.

\begin{ex}\label{ex_versals}
Consider the smooth map-germ $f\colon(\C,0)\to(\C^2,0)$ given by $f(y) = (y^2,y^{2p+1})$, for any integer $p \geq 1$.
It admits an OPSU of the form $F(x,t) = (y^2,y^{2p+1}+ty,t)$.
Consider a smooth function $g\colon (\C,0)\to(\C,0)$ of the form $g(z) = z^q$ for some integer $q\geq 2$.

We have that $\nae f = \operatorname{Sp}_\C\left\lbrace (0,y), (0,y^3),\ldots,(0,y^{2p-1})\right\rbrace$, and that $\frac{\O_1}{\langle g\rangle + Jg}$ admits a monomial $\C$-basis formed by $1, z, z^2,\ldots, z^{q-2}$.
Then, by Corollary \ref{coro_nice_versal}, the augmentation:
$$A_{F,g}(f)(y,z) = (y^2,y^{2p+1} + z^qy,z)$$
admits the following versal unfolding:
$$VA_{F,g}(f)(y,z,\lambda) = \left(y^2, y^{2p+1} + z^qy + \sum_{s=1}^{q-1}\sum_{m=1}^{p} \lambda_{s,m}z^{s-1}y^{2m-1}, z, \lambda\right)$$
Compare these augmentations with Mond's list of simple maps from $\C^2$ to $\C^3$, \cite{mond}.
When $p = 1$, then $A_{F,g}(f)\equiv S_q$.
If $p\geq 1$ and $q = 2$, we have $A_{F,g}(f)\equiv B_p$.
And finally, if $p=2, q=3$, $A_{F,g}(f) \equiv F_4$.
In all the other cases, the augmentation is non-simple.

\end{ex}

\section{Simplicity of Augmentations}\label{simplicity}

We prove in this section our main results about simplicity of an augmentation.
Recall we are considering $f\colon(\C^n,0)\to(\C^p,0)$ a holomorphic map-germ with $(n,p)$ in the nice dimensions, that admits an OPSU $F(x,t)=(f_t(x),t)$, and $g\colon(\C^d,0)\to(\C,0)$ a smooth function.

Our first result is a condition for the lack of simplicity depending only on the augmenting function:

\begin{teo}\label{thm_nonsimple}
$$g \text{ not } \mathscr R\text{-simple} \implies A_{F,g}(f)  \text{ not } \A\text{-simple} $$
\end{teo}

\begin{proof}
Arnold proved in \cite{arnoldnf} that every non-simple function is adjacent to one of the following non-simple families:
\begin{equation*}
\begin{aligned}
P_8 & \equiv x^3+y^3+z^3+\lambda xyz   && \lambda^3 + 27 \neq 0 \\
X_9 & \equiv x^4+y^4 + \lambda x^2y^2  && \lambda^2 \neq 4 \\
J_{10} & \equiv x^3+y^6+\lambda x^2y^2 && 4\lambda^3 + 27\neq 0
\end{aligned}
\end{equation*}
Therefore it is sufficient to see that augmenting via any of these functions results in a non-simple augmentation.
Denote any of the three families by $g_\lambda$, then $A_{F,g_\lambda}(f)$ is a 1-parameter family of deformations of $A_{F,g_0}(f)$.

Note that each $g_\lambda$ is quasihomogeneous, and $\mu(g_\lambda) = \mu(g_0)< \infty$.
Since $g_\lambda$ is not $\mathscr R$-equivalent to $g_{\lambda'}$ for every $\lambda \neq \lambda'$, we can apply Theorem \ref{thm_r_iff_a} to see that $A_{F,g_\lambda}(f)$ is not $\mathcal{A}$-equivalent to $A_{F,g_{\lambda'}}(f)$ for every $\lambda\neq\lambda'$.
So $A_{F,g_0}(f)$ cannot be simple.

\end{proof}

The rest of this section will be focused on characterising simplicity in the case that $\Aecod(f) = 1$ or that $g$ is a Morse function.
What makes these cases special is the type of germs that appear in the bifurcation set of the augmentation: every deformation of the augmentation is an augmentation of a deformation of $f$ by a deformation of $g$.

\begin{prop}\label{prop_todo_aug}
Let $f\colon(\C^n,0)\to(\C^p,0)$ be a smooth map-germ with $\Aecod(f) = k$ that admits an OPSU $F$ and consider $g\in \O_d$ a smooth simple function with $\tau(g)=r$.

If $\Aecod(f) = 1$ or $g$ is a Morse function, there exist versal unfoldings $(f_\lambda,\lambda)$ and $(g_\mu,\mu)$ of $f$ and $g$ respectively such that any deformation in the versal unfolding of the augmentation satisfies $\left(A_{F, g}(f)\right)_\gamma = A_{F,g_\mu}(f_\lambda)$ for some $\lambda\in\C^k$ and $\mu\in\C^r$.
\end{prop}

\begin{proof}
Write $F(x,t) = (\tilde{f}_t(x),t)$. 
Choosing appropiate $\gamma_2,\ldots,\gamma_k\in\theta(f)$ and $\tau_1,\ldots,\tau_r\in\O_d$, and writing:
\begin{align*}
f_\lambda(x) & = \tilde{f}_{\lambda_1}(x)+\sum_{m=2}^k\lambda_m\gamma_m(x)\\
g_\mu(z) & = g(z) + \sum_{s=1}^r\mu_s\tau_s(z)
\end{align*}
we can obtain versal unfoldings $(f_\lambda(x),\lambda)$ and $(g_\mu(z),\mu)$ of $f$ and $g$ respectively.
Moreover, we can take $\tau_r(z) = 1$.
Then, since $g$ is simple and therefore $\mathscr R$-equivalent to a quasihomogeneous function, we can apply Corollary \ref{coro_nice_versal} to obtain versal unfoldings in each case.

First assume $\Aecod(f) = 1$, denote for each $\lambda\in\C$ the unfolding $F_\lambda(x,t) = (\tilde{f}_{t + \lambda}(x),t)$, which is an OPSU for $f_\lambda(x)$.
Then the versal unfolding of $A_{F,g}(f)$ takes the form:
$$VA_{F,g}(f)(x,z,\mu) = \left(\tilde{f}_{g(z) +\sum_{s=1}^{r-1}\mu_s\tau_s(z) + \mu_r}(x),z,\mu\right) = \left(A_{F_{\mu_r}, g_\mu(z) - \mu_r}(\tilde f_{\mu_r}),\mu\right)$$
That is: every deformation of $A_{F,g}(f)$ is $\A$-equivalent, for some $\mu\in\C^r$, to the augmentation of $\tilde{f}_{\mu_r}$ by $F_{\mu_r}$ via the augmenting function $g_\mu(z) - \mu_r$.

Finally, if $g$ is a Morse function, that is $r = 1$, then denote for each $\lambda\in\C^k$ the unfolding:
$$F_\lambda(x,t) = \left(\tilde{f}_{t + \lambda_1}(x) + \sum_{m=2}^{k}\lambda_m\gamma_m(x),t\right)$$
which is an OPSU for $f_\lambda$.
Then, the versal unfolding of $A_{F,g}(f)$ takes the form:
$$VA_{F,g}(f)(x,z,\lambda) = \left(\tilde{f}_{g(z) + \lambda_1}(x) + \sum_{m=2}^{k}\lambda_m\gamma_m(x),z,\lambda\right) = \left(A_{F_\lambda,g}(f_\lambda),\lambda\right)$$
That is, every deformation of $A_{F,g}(f)$ is $\A$-equivalent to the augmentation of some $f_\lambda$ via a Morse function, for some $\lambda\in\C^k$.

\end{proof}

Notice that a converse result also holds: every augmentation of the germ of a deformation of $f$, appears as a deformation of the augmentation $A_{F,g}(f)$ (see Remark \ref{rem_marartari}).

\begin{teo}\label{thm_cod1_simple}
Assume $\Aecod(f) = 1$. Then:
\begin{equation*}
A_{F,g}(f)   \text{ is } \A\text{-simple}  \iff g  \text{ is } \mathscr R\text{-simple}
\end{equation*}
\end{teo}
\begin{proof}
Assume $A_{F,g}(f)$ is $\A$-simple: then, by Theorem \ref{thm_nonsimple}, $g$ is $\mathscr R$-simple.

For the opposite implication, take a function $g$ that is $\mathscr R$-simple.
Since $\Aecod(f)=1$, $f$ admits a versal unfolding of the form:
$$F(x,t) = (f(x)+t\gamma(x),t)$$
for some $\gamma\in\theta(f)$.
By Theorem \ref{thm_cod1_opsu}, we only need to check the result using $F$.

Using Proposition \ref{prop_todo_aug}, we have that the germ of any deformation in the versal unfolding has the form:
$$VA_{F,g}^\lambda(f)(x,z) =  \left(f(x) + \left(g(z) + \sum_{i=1}^r\lambda_i\tau_i(z)+  \lambda_{r+1} \right)\gamma(x),z\right)$$
for small enough $\lambda = (\lambda_1,\ldots,\lambda_{r+1})\in\C^r$.
Here, $\tau(g) = r+1$ and $\tau_1,\ldots,\tau_r\in \O_d$ are chosen so that $\text{Sp}_\C\left\lbrace 1,\tau_1,\ldots,\tau_r\right\rbrace = \frac{\O_d}{\langle g\rangle + Jg}$.
Denote $g_\lambda(z) = g(z) + \sum_{i=1}^r \lambda_i\tau_i(z)$.

If $\lambda_{r+1} = 0$, $VA_{F,g}^\lambda(f)$ is of the form $A_{F,g_\lambda}(f)$.
Here we only find a finite number of left-right equivalence classes, since $g$ is $\mathscr R$-simple and we can apply Theorem \ref{thm_r_iff_a}.

If $\lambda_{r+1}\neq 0$, $VA_{F,g}^\lambda(f)$ is the augmentation of $f(x) +\lambda_{r+1}\gamma(x)$ via $g_\lambda$.
Since $\Aecod(f) = 1$ and the pair $(n,p)$ is in the nice dimensions, $f$ is simple.
Therefore, since $F$ is a versal unfolding, $f(x) + \lambda_{r+1}\gamma(x)$ is stable.
Hence, $VA^\lambda_{F,g}(f)$ is an unfolding of a stable germ and so it is stable.

Taking a representative of the versal unfolding on small enough neighbourhoods of the origin, both cases for $\lambda_{r+1}$ will only add a finite number of left-right equivalence classes and so $A_{F,g}(f)$ is simple.

\end{proof}

\begin{rem}
The above proof could be completed exactly the same way with a generic OPSU instead of using Theorem \ref{thm_cod1_opsu}, but this way results in a cleaner proof.
\end{rem}

\begin{coro}\label{coro_modality}
If $\Aecod(f) = 1$ and all the deformations of $g$ along its $\mu$-constant stratum are quasihomogeneous, then the modality of $A_{F,g}(f)$ is equal to the modality of $g$.
\end{coro}
\begin{proof}
By Theorem \ref{thm_cod1_simple}, the points in the parameter space of $A_{F,g}(f)$ for which the deformation is not simple are the same that belong to the $\mu$-constant stratum of $g$.
Therefore, their modality is the same.
\end{proof}

This allows us to produce examples of finitely determined non-simple map-germs with the same modality as the augmenting function.
\begin{ex}
Consider $f(z) = z^3$ and the augmenting unimodal function $J_{10}$ of the form $g(x,y)=x^3+y^6 +ax^2y^2$, with $4a^3+27\neq0$.
Then the augmentation: $$A_{F,g}(f)=(x,y,z^3+(x^3+y^6+ax^2y^2)z)$$ has $\mathscr A_e$-codimension $10$ and is unimodal by Corollary \ref{coro_modality}
\end{ex}

The last theorem in this section is the characterization for the simplicity of augmentations via Morse augmenting functions:

\begin{teo}\label{thm_morse_simple}
If $g$ is a Morse function, then:
\begin{equation*}
A_{F,g}(f)  \text{ is } \A\text{-simple}  \iff  f  \text{ is } \A\text{-simple}
\end{equation*}
\end{teo}
\begin{proof}
It is enough to prove the result for the case that $g(z) = z^2$, since any augmentation via a Morse function on serveral variables will be equivalent to the repeated agumentation of a Morse function in one variable.

Let $f$ be a simple map-germ. If $\Aecod(f) = 0$, then the augmentation is stable and therefore simple.
We proceed by induction: assume for some positive integer $d$ that, given any simple $f$ such that $\Aecod(f) \leq d$, $A_{F,g}(f)$ is simple.

Given any $f$ with $\Aecod(f)= d+1$, let $\bar F(x,\lambda) = (\bar f_\lambda(x),\lambda)$ be its versal unfolding. Then $\Aecod(\bar f_\lambda) \leq d$ since $f$ is simple.
Now, Proposition \ref{prop_todo_aug} ensures the germ of any deformation in the versal unfolding of $A_{F,g}(f)$ is the augmentation of some $\bar f_\lambda$ via $g$, which are all simple by the induction hypothesis.
This is only possible if $A_{F,g}(f)$ is simple.

Assume now that $f$ verifies $A_{F,g}(f)$ is simple.
If $\Aecod(A_{F,g}(f)) = 0$ we get by Corollary \ref{coro_igualdad_houston} that $\Aecod(f) = 0$.
Hence $f$ is simple.
Proceeding again by induction we assume that for some positive integer $d$, if $f$ verifies that $A_{F,g}(f)$ is simple and $\Aecod(A_{F,g}(f)) \leq d$, $f$ is simple.

Suppose $A_{F,g}(f)$ is simple and $\Aecod(A_{F,g}(f)) = d+1$. 
By Proposition \ref{prop_todo_aug} the germ of any deformation in its versal unfolding is the augmentation of some $\bar f_\lambda$ via $g$.
Conversely, the augmentation of any deformation of $f$ via $g$ is a deformation of $A_{F,g}(f)$.
In particular, since $A_{F,g}(f)$ is simple, each of those augmentations has codimension $d$ or lower.
Applying the induction hypothesis, for any small enough $\lambda\neq 0$ we have that $\bar f_\lambda$ is simple.
This implies that $f$ is simple.

\end{proof}

These results allow us to easily produce simple germs, such as those map-germs in already existing classifications that turn out to be augmentations in the conditions of the above theorems.

\begin{ex}\label{ex_simples}
As we saw on Example \ref{ex_versals}, all $S_k$ and $B_k$ from Mond's list of map-germs from $\C^2$ to $\C^3$ (\cite{mond}) are augmentations of $\A_e$-codimension 1 curves or obtained via a Morse augmenting function.
Therefore, by Theorems \ref{thm_cod1_simple} and \ref{thm_morse_simple} we have that they are simple.

Taken from Marar and Tari's classification of simple map-germs from $\C^3$ to $\C^3$ (\cite{marartari}), the following are augmentations:
\begin{equation*}
\begin{aligned}
3_Q & \equiv (x,y,z^3+Q(x,y)z) && Q \text{ simple }\\
4_1^k & \equiv (x,y,z^4+xz\pm y^kz^2) && k\geq 1\\
4_2^k & \equiv (x,y,z^4+(y^2\pm x^k)z + xz^2)  && k\geq 2
\end{aligned}
\end{equation*}
Then $3_Q$ is the augmentation of $z^3$, which has $\A_e$-codimension 1 via a simple function, and so is simple by Theorem \ref{thm_cod1_simple}.
$4_1^k$ is the augmentation of $(x,z^4+xz)$, of codimension $1$, via $y^k$, a simple function, and so is simple by Theorem \ref{thm_cod1_simple}.
Finally, $4_2^k$ is the augmentation of $(x,z^4 + x^kz + xz^2)$ via a Morse function, and so it is simple by Theorem \ref{thm_morse_simple}
\end{ex}

\begin{rem}\label{rem_marartari}
The last example shows an adjacency is missing in Marar and Tari's classification.
It can be checked that $(x,z^4 + x^2z + xz^2)$ is $\A$-equivalent to the map-germ $11_5$ in Rieger's classification (\cite{rieger2en2}).
The adjacency diagrams in this reference show that $11_5$ is adjacent to the lips-beaks singularity $LB_2 \equiv (x,y^3\pm x^2y)$.
Therefore, by Proposition \ref{prop_todo_aug}, the augmentation of $11_5$ via $g(z) = z^2$, (namely $4_2^2$ in Marar and Tari's list) must be adjacent to the augmentation of $LB_2$ via $g$ (namely, $3_{A_1}$ in Marar and Tari's list):
Thus, the arrow:
$$ 3_{A_1} \longleftarrow 4_2^2 $$ should be added to the last diagram shown in \cite{marartari}.

It is interesting to notice that there are no more adjacencies between $11_{2k+1}$ from Rieger's list and $LB_k$ for $k \geq 3$, and so no other adjacencies are missing in the diagram.

\end{rem}

\begin{ex}
When augmenting map-germs of $\A_e$-codimension greater than 1, via  non-Morse functions, simplicity of the augmentation does not seem to be directly related to the simplicity of the base map-germ or the augmenting function.
Consider $f\colon(\C,0)\to(\C^2,0)$ given by $f(y) = (y^2,y^5)$, which has $\Aecod(f) = 2$ and admits an OPSU $F(y,t) = (y^2,y^5 +ty,t)$.

Augmenting $f$ via $g_1(z) = z^3$ produces the following map-germ, present in Mond's list of simple map-germs, \cite{mond}:
$$F_4\equiv (y^2,y^5+z^3y,z)$$
But augmenting the same $f$ via another simple function, $g_2(z) = z^4$ returns a map-germ not equivalent to any of the germs in the list, and so it is not simple:
$$(y^2,y^5+z^4y,z)$$

\end{ex}

\section{Simple corank 1 augmentations from $\C^4\to\C^4$}\label{4en4}

In this section we apply our results to obtain a list of simple corank 1 augmentations in $\C^4\to\C^4$ from $\A_e$-codimension 1 germs and by Morse functions, shown in Table \ref{table_44}. 
Here, superindexes represent the $\A_e$-codimension of the augmented map-germ, and subindexes represent the augmenting function.

\bgroup
\def\arraystretch{1.2}%
\begin{table}[h]
\begin{tabular}{clcl}
\hline
Type    & \multicolumn{1}{c}{Normal Form}    & \multicolumn{1}{c}{$\A_e$-codimension} &            \\ \hline
$3_{P}$ & $(x,y,z,t^3+P(x,y,z)t)$            & $\mu(P)$                               & \\
$4_Q$   & $(x,y,z,t^4+xt+Q(y,z)t^2)$         & $\mu(Q)$                               &            \\
$4_k^2$ & $(x,y,z,t^4+(x^k+y^2+z^2)t+xt^2)$   & $k$                                    & $k\geq 2$  \\
$5_k$   & $(x,y,z,t^5+xt+yt^2+z^kt^3)$       & $k-1$                                  & $k \geq 1$ \\
$5^2$   & $(x,y,z,t^5+xt+(y^2+z^2)t^2+yt^3)$ & $2$                                    &            \\
$5^3$   & $(x,y,z,t^5+xt+z^2t^2+yt^3)$       & $3$                                    &            \\ \hline
\end{tabular}
\medskip
\caption{$P,Q$ are of type $A_k,D_k,E_6,E_7,E_8$.}
\label{table_44}
\end{table}
\egroup

\begin{teo}
If a corank 1 map-germ $A\colon(\C^4,0)\to(\C^4,0)$ is the result of augmenting an $\A_e$-codimension 1 map-germ, or is an augmentation via a Morse function, then $A$ lies in the equivalence class of one of the germs in Table \ref{table_44}.
\end{teo}

\begin{proof}
Let $f\colon(\C^n,0)\to(\C^p,0)$ be a simple smooth map-germ of corank 1 which admits an OPSU $F$.
Let $g:(\C^d,0)\to(\C,0)$ be a simple function.
Write $A(x,z) = A_{F,g}(f)(x,z)$ to simplify.
We argue on all the possible dimensions that $f$ can take.

\textbf{Case 1:} $(n,p) = (1,1)$.
Here $f(t) = t^3$ is the only germ that admits an OPSU, which can be of the form $F(t,\lambda) = (t^3+\lambda t,\lambda)$.
To match the dimensions of $A$, it is necessary that $d=3$, so $g$ is a simple function on three variables.
Therefore, $A$ lies in the class of $3_g$.

\textbf{Case 2:} $(n,p) = (2,2)$.
From Rieger's classification, \cite{rieger2en2}, we have that $f$ only admits an OPSU if it lies in the class of one of the following:
\begin{align*}
LB_k & \equiv (x,t^3+x^kt), \; k \geq 2\\
S &\equiv (x,t^4 +xt)\\
11_{2k+1} & \equiv (x,t^4+x^kt+xt^2), \; k\geq 2
\end{align*}
Here, $d=2$. Since all $LB_k$ are agumentations of $t^3$, any augmented $LB_k$ lies in the class of $3_P$ with $P$ a two-variable simple function if $k=2$, or a Morse function if $k>2$.

$S$ has $\A_e$-codimension 1, and admits an OPSU $(x,\lambda,t^4+xt +\lambda t^2)$.
Here the augmenting function can be any 2-variable simple function $Q$, and so $A$ lies in the class of $4_Q$.
Finally, $11_{2k+1}$  for any $k\geq 2$ can only be augmented by a 2-variable Morse function, and so it lies on the class of $4^2_k$.

\textbf{Case 3:} $(n,p)=(3,3)$.
We turn to Marar and Tari's classification, \cite{marartari}, and check that all the germs present admit an OPSU.
In their notation, all of $4_1^k, 4_2^k$ and the family of the form $(x,y,z^3 + \tilde Q(x,y)z)$ with $\tilde Q$ simple are augmentations, and therefore it is easy to check that their augmentations lie on the classes of $4_Q,4^2_k$ and $3_P$ respectively with our notation.
The other classes are primitive, and in their notation take the form:
\begin{align*}
5_1 & \equiv (x,y,t^5+xt+yt^2) \\
5_2 & \equiv (x,y,t^5+xt+y^2t^2+yt^3) \\
5_3 & \equiv (x,y,t^5 + xt + yt^3)
\end{align*}
As $d=1$, if $f\equiv 5_1$, it can only be augmented via $A_k$ functions and so $A$ would lie, in our notation, in the $5_k$ classes.
The classes $5_2$ and $5_3$ in their notation have $\A_e$-codimension greater than one, and therefore we can only augment them via a Morse function, after which it is easy to check that $A$ would lie on our $5^2$ and $5^3$ classes respectively.

Notice that Theorems \ref{thm_cod1_simple} and \ref{thm_morse_simple} ensure that all of the germs in Table \ref{table_44} are simple.

\end{proof}

It can be easily checked that all corank 1 augmentations in $\C^m\to\C^m$ for $m=2,3$ are either augmentations of an $\A_e$-codimension 1 map-germ, or augmentations via Morse functions.
We conjecture these are also all the possible simple augmentations of corank 1 that appear in $\C^4\to\C^4$.
This can probably be checked using Goryunov's method in \cite{gor84}.

\end{document}